\documentclass{article}

\usepackage{vmargin}
\setmarginsrb{1.1in}{1in}{1.1in}{1in}%
             {0.1in}{0.1in}{0.1in}{0.3in}

\usepackage[english]{babel}
\usepackage{amsmath,amsfonts,amssymb,amsthm}
\usepackage{latexsym}
\usepackage{makeidx}

\numberwithin{equation}{section}

\newtheorem{em-deff}{Definition}[section]
\newtheorem{lemma}[em-deff]{Lemma}
\newtheorem{theorem}[em-deff]{Theorem}
\newtheorem{corollary}[em-deff]{Corollary}

\newtheorem{proposition}[em-deff]{Proposition}
\newtheorem{em-fact}[em-deff]{Fact}
\newtheorem{em-example}[em-deff]{Example}

\newtheorem{problem}[em-deff]{Problem}

\newtheorem{em-remark}[em-deff]{Remark}

\newenvironment{example}{\begin{em-example} \em }{ \end{em-example}}
\newenvironment{remark}{\begin{em-remark} \em }{ \end{em-remark}}
\newenvironment{deff}{\begin{em-deff} \em }{ \end{em-deff}}

\newcommand{\N}{\mathbb N}
\newcommand{\Z}{\mathbb Z}

\def\ent{\mathrm{ent}}
\def\supp{\mathrm{supp}}
\def\Per{\mathrm{Per}}

\input xy
\xyoption{all}

\title{Algebraic entropy of generalized shifts on direct products}
\author{Anna Giordano Bruno\\ \small{Dipartimento di Matematica Pura e Applicata} \\ \small{Universit\`a di Padova} \\ \small{Via Trieste, 63 - 35121 Padova, Italy}\\ \small{anna.giordanobruno@math.unipd.it}}
\date{}

\begin{document}

\maketitle

\begin{abstract}
For a set $\Gamma$, a function $\lambda:\Gamma\to \Gamma$ and a non-trivial abelian group $K$, the \emph{generalized shift} $\sigma_\lambda:K^\Gamma\to K^\Gamma$ is defined by $(x_i)_{i\in \Gamma}\mapsto (x_{\lambda(i)})_{i\in\Gamma}$ \cite{AKH}.
In this paper we compute the algebraic entropy of $\sigma_\lambda$; it is either zero or infinite, depending exclusively on the properties of $\lambda$. This solves two problems posed in \cite{AADGH}.
\end{abstract}

\section{Introduction}

The general aim of this paper is to study the algebraic entropy of relevant endomorphisms of abelian groups, as generalized shifts are, in view of the recent results in \cite{AADGH} and \cite{DGSZ}.

\medskip
According to Adler, Konheim and McAndrew \cite{AKM} and Weiss \cite{W} the algebraic entropy is defined as follows. Let $G$ be an abelian group and $F$ a finite subgroup of $G$; for an endomorphism $\phi:G\to G$ and a positive integer $n$, let $T_n(\phi,F)=F+\phi(F)+\ldots+\phi^{n-1}(F)$ be the \emph{$n$-th $\phi$-trajectory} of $F$ with respect to $\phi$. The \emph{algebraic entropy of $\phi$ with respect to $F$} is $$H(\phi,F)={\lim_{n\to \infty}\frac{\log|T_n(\phi,F)|}{n}},$$ and the \emph{algebraic entropy} of $\phi:G\to G$ is $$\ent(\phi)=\sup\{H(\phi,F): F\ \text{is a finite subgroup of } G\}.$$

In Section \ref{background} we collect the general results on the algebraic entropy that we use in this paper, including the so-called Addition Theorem from \cite{DGSZ} (see Theorem \ref{AT} below).

\medskip
In {\cite{AKH}} the notion of generalized shift was introduced as follows.

\begin{deff}
Let $\Gamma$ be a set, $\lambda:\Gamma\to\Gamma$ a function and $K$ an abelian group. The \emph{generalized shift} $\sigma_{\lambda,K}:K^\Gamma\to K^\Gamma$ is defined by $(x_i)_{i\in \Gamma}\mapsto (x_{\lambda(i)})_{i\in\Gamma}$. When there is no need to specify the group $K$, we simply write $\sigma_\lambda$.
\end{deff}

In Section \ref{background} we give basic properties of the generalized shifts.

The interest in studying the generalized shifts arises from the fact that there is a close relation between the generalized shifts and the Bernoulli shifts: let $K$ be a non-trivial finite abelian group, and denote by $\N$ and $\Z$ respectively the set of natural numbers and the set of integers; then:
\begin{itemize}
\item[(a)] the \emph{two-sided Bernoulli shift} $\overline{\beta}_K$ of the group $K^{\Z}$ is defined by 
$$\overline\beta_K((x_n)_{n\in\Z})=(x_{n-1})_{n\in\Z}, \mbox{ for } (x_n)_{n\in\Z}\in K^{\Z};$$
\item[(b)] the \emph{right  Bernoulli shift} $\beta_K$ and the \emph{left Bernoulli shift} $_K\beta$ of the group $K^{\N}$ are defined respectively by 
$$\beta_K(x_1,x_2,x_3,\ldots)=(0,x_1,x_2,\ldots)\ \mbox{and}\ _K\beta(x_0,x_1,x_2,\ldots)=(x_1,x_2,x_3,\ldots).$$
\end{itemize}
The left Bernoulli shift $_K\beta$ and the two-sided Bernoulli shift $\overline\beta_K$ are relevant for both ergodic theory and topological dynamics and they are generalized shifts (see Example \ref{bernoulli}). The right Bernoulli shift $\beta_K$ restricted to the direct sum $\bigoplus_\N K$ is fundamental for the algebraic entropy (see \cite{DGSZ}). It cannot be obtained as a generalized shift from any function $\lambda:\N\to\N$; nevertheless, it can be well ``approximated" by a generalized shift \cite{AADGH} (see Example \ref{bernoulli}). 

\medskip
In \cite{AADGH} the restriction of a generalized shift $\sigma_\lambda$ to the direct sum $\bigoplus_\Gamma K$ was considered. Indeed, a precise formula was found for the algebraic entropy of this restriction (see \eqref{aadgh-eq} in Theorem \ref{mt-sum}). In particular, the algebraic entropy of $\sigma_\lambda\restriction_{\bigoplus_\Gamma K}$ depends on the combinatorial invariant that measures the number of strings of $\lambda$ (see Definitions \ref{sol} and \ref{sol2}) and on the cardinality of $K$. Note that in this case $\lambda$ must have finite fibers in order that $\bigoplus_\Gamma K$ is $\sigma_\lambda$-invariant.
 
\medskip
Problems 6.1 and 6.2 in \cite{AADGH} ask to calculate the algebraic entropy of $\sigma_\lambda:K^\Gamma\to K^\Gamma$ and to relate this entropy with the algebraic entropy of $\sigma_\lambda\restriction_{\bigoplus_\Gamma K}:\bigoplus_\Gamma K\to \bigoplus_\Gamma K$. We provide a complete answer to these questions.
More precisely, we show that $\ent(\sigma_{\lambda})$ depends only on the combinatorial properties of the map $\lambda$, unlike $\ent(\sigma_\lambda\restriction_{\bigoplus_\Gamma K})$. Indeed, Theorem \ref{mt} shows that $\ent(\sigma_{\lambda})=0$ if and only if $\lambda$ is bounded (in the sense of the next Definition \ref{bounded-def}), otherwise $\ent(\sigma_{\lambda})$ is infinite. 

\smallskip
The function $\lambda:\Gamma\to \Gamma$ of a set $\Gamma$ defines a preorder $\leq_\lambda$ on $\Gamma$ in a natural way: $i\leq_\lambda j$ in $\Gamma$ if there exists $s\in\N$ such that $\lambda^s(j)=i$.
The preorder $\leq_\lambda$ is not an order in general: two distinct elements $i$ and $j$ of $\Gamma$ violate the antisymmetry for $\leq_\lambda$ if and only if $i$ and $j$ are in the same orbit of a periodic point (which could be $i$ or $j$) of $\lambda$.
We say that a subset $I$ of $(\Gamma,\leq_\lambda)$ is \emph{totally preordered} if for every $i,j\in I$ either $i\leq_\lambda j$ or $j\leq_\lambda i$ (without asking that these elements satisfy antisymmetry).

\begin{deff}\label{bounded-def}
Let $\Gamma$ be a set. A function $\lambda:\Gamma\to\Gamma$ is \emph{bounded} if there exists $N\in\N$ such that $|I|\leq N$ for every totally preordered subset $I$ of $(\Gamma,\leq_\lambda)$. 
\end{deff}

In Section \ref{strings} we analyze the properties of a function $\lambda$ which play a role with respect to the algebraic entropy of $\sigma_\lambda$ and find characterizations of bounded functions. Indeed, Theorem \ref{bounded} shows that a function is bounded if and only if it admits no strings, no infinite orbits and no ladders, and has bounded periodic orbits; here, a string of $\lambda$ is an infinite increasing chain in $(\Gamma,\leq_\lambda)$, while an infinite orbit is an infinite decreasing chain in $(\Gamma,\leq_\lambda)$ and a ladder of $\lambda$ is a disjoint union of infinitely many finite chains in $(\Gamma,\leq_\lambda)$ of strictly increasing length where the top element of each finite chain is a maximal element (for the precise definitions of these notions see Definition \ref{sol}).

Theorem \ref{bounded} proves also that for a function $\lambda$ it is equivalent to be bounded or quasi-periodic (the definition is given below).

\smallskip
For a function $f:X\to X$ of a set $X$, a point $x\in X$ is said to be \emph{quasi-periodic} if there exist $n_x<m_x$ in $\N$ such that $f^{n_x}(x)=f^{m_x}(x)$. The function $f$ is \emph{locally quasi-periodic} if every point of $X$ is quasi-periodic, and $f$ is \emph{quasi-periodic} if there exist $n<m$ in $\N$ such that $n_x=n$ and $m_x=m$ for every $x\in X$, that is, $f^n=f^m$.

\smallskip
It is known from \cite{DGSZ} that $\ent(\phi)=0$ if and only if $\phi$ is locally quasi-periodic.
The main goal of this paper is to prove the following theorem, showing in particular that for generalized shifts this ``local" condition becomes ``global", and that a generalized shift of finite algebraic entropy has necessarily entropy zero.

\begin{theorem}\label{mt}
Let $\Gamma$ be a set, $\lambda:\Gamma\to \Gamma$ a function, $K$ a non-trivial finite abelian group, and $\sigma_\lambda:K^\Gamma\to K^\Gamma$ the generalized shift. The following conditions are equivalent:
\begin{itemize}
\item[(a)] $\ent(\sigma_\lambda)=0$;
\item[(b)] $\ent(\sigma_\lambda)$ is finite;
\item[(c)] $\lambda$ is bounded;
\item[(d)] $\sigma_\lambda$ is quasi-periodic;
\item[(e)] $\sigma_\lambda$ is locally quasi-periodic.
\end{itemize}
\end{theorem}

At the end of Section \ref{background} we see that $\lambda$ is quasi-periodic if and only if $\sigma_\lambda$ is quasi-periodic.
In Section \ref{qp=lqp} we prove first the equivalence of (c), (d) and (e), without involving the algebraic entropy (see Theorem \ref{pre-mt}), even if, as noted previously, (a)$\Leftrightarrow$(e) is already known from \cite{DGSZ} (see Proposition \ref{properties}(a)).

The implication (a)$\Rightarrow$(b) is obvious. The main part of the paper is dedicated to the proof of (b)$\Rightarrow$(a), that is, to prove that if $\ent(\sigma_\lambda)$ is positive, then $\ent(\sigma_\lambda)$ is infinite. Given the above equivalence (a)$\Leftrightarrow$(c), this is the same as proving (b)$\Rightarrow$(c), which is what we verify.

Section \ref{disjoint} contains technical lemmas, which allow the construction of large independent families of finite subgroups of $K^\N$. These families of subgroups are used in the computation of the algebraic entropy of a generalized shift $\sigma_\lambda$ when $\lambda$ admits some string or some infinite orbit. 

Indeed, in Section \ref{infinite} we prove that in presence of a string or of an infinite orbit of $\lambda$, the algebraic entropy of $\sigma_\lambda$ is infinite. The same happens if $\lambda$ has a ladder or if it has periodic orbits of arbitrarily large length. This shows that the entropy of $\sigma_\lambda$ is infinite when the function $\lambda$ is not bounded. 

After the proof of Theorem \ref{mt} we explain how it solves Problems 6.1 and 6.2 in \cite{AADGH}.

\medskip
As applications, in Corollary \ref{beta} we see that the algebraic entropy of the Bernoulli shifts considered on the direct products is infinite, and in Corollary \ref{recurrence} we strengthen a result from \cite{DGSZ}, related to the Poincar\'e--Birkhoff recurrence theorem of ergodic theory, in the particular case of the generalized shifts.

\subsection*{Acknowledgements}

I am grateful to Professor Dikran Dikranjan and Professor Luigi Salce for their useful comments and suggestions.
I would like to express also my thanks to the referee for his/her suggestions.

\section{Preliminary results}\label{background}

We start collecting basic results on the algebraic entropy, mainly from \cite{DGSZ} and \cite{W}, which are applied in the sequel.

\medskip
First of all, since the definition of the algebraic entropy of an endomorphism $\phi$ of an abelian group $G$ is based on the finite subgroups $F$ of $G$, the algebraic entropy depends only on the restriction of $\phi$ on $t(G)$, that is $\ent(\phi)=\ent(\phi\restriction_{t(G)})$. So it makes sense to consider endomorphisms of torsion abelian groups.

\smallskip
Let $G$ be a torsion abelian group, $\phi:G\to G$ an endomorphism and $H$ a $\phi$-invariant subgroup of $G$. Denote by $\overline{\phi}:G/H\to G/H$ the endomorphism induced on the quotient by $\phi$. Then $\ent(\phi)\geq\max\{\ent(\phi\restriction_H),\ent(\overline\phi)\}.$
Moreover, the following important result on the algebraic entropy holds true:

\begin{theorem}[Addition Theorem]\label{AT}\emph{\cite[Theorem 3.1]{DGSZ}}
Let $G$ be a torsion abelian group, $\phi:G\to G$ an endomorphism and $H$ a $\phi$-invariant subgroup of $G$. If $\overline{\phi}:G/H\to G/H$ is the endomorphism induced on the quotient by $\phi$, then $$\ent(\phi)=\ent(\phi\restriction_H)+\ent(\overline\phi).$$
\end{theorem}

For an endomorphism $\phi$ of a torsion abelian group $G$ and a finite subgroup $F$ of $G$, the \emph{$\phi$-trajectory} of $F$ is $T(\phi,F)=\sum_{n\in\N}\phi^n(F)$.   Let $$t_\phi(G)=\{x\in G:|T(\phi,\langle x\rangle)|\ \text{is finite}\}$$ be the \emph{$\phi$-torsion} subgroup of $G$.
Then $t_\phi(G)$ is the largest $\phi$-invariant subgroup of $G$ such that $\ent(\phi\restriction_{t_\phi(G)})=0$. 
In particular every quasi-periodic point $x$ of $\phi$ in $G$ has finite trajectory and so it is in $t_\phi(G)$.

\medskip
We collect in the following proposition the basic and well-known results on the algebraic entropy that we will use in the paper; for a proof of (a), (b) and (c) see \cite{DGSZ} and \cite{W}, while (d) can be derived from (c), from the finite case and from the monotonicity of the algebraic entropy under taking invariant subgroups, both proved in \cite{W}.

\begin{proposition}\label{properties}
Let $G$ be a torsion abelian group and $\phi:G\to G$ an endomorphism. Then:
\begin{itemize}
\item[(a)] $\ent(\phi)=0$ if and only if $t_\phi(G)=G$ if and only if $\phi$ is locally quasi-periodic.
\item[(b)] If $H$ is another abelian group, $\eta:H\to H$ an endomorphism, and there exists an isomorphism $\xi:G\to H$ such that $\phi=\xi^{-1}\eta\xi$, then $\ent(\phi)=\ent(\eta)$.
\item[(c)]If $G$ is direct limit of $\phi$-invariant subgroups $\{G_i:i\in I\}$, then ${\ent}(\phi)=\sup_{i\in I}\ent(\phi {\restriction_{G_i}})$.
\item[(d)]If $G=\prod_{i\in I} G_i$, where each $G_i$ is a $\phi$-invariant subgroup of $G$, then $\ent(\phi)\geq\sum_{i\in I}\ent(\phi\restriction_{G_i})$.
\end{itemize}
\end{proposition}

\medskip
Now we summarize the preliminary results on the generalized shifts, recalling in Proposition \ref{composition} some basic facts which are mostly proved in \cite{AADGH} and \cite{AKH}.

\smallskip
Let $\Gamma$ be a set and $\lambda:\Gamma\to\Gamma$ a function. If $K$ is an abelian group, the support of an element $x=(x_i)_{i\in\Gamma}$ of $K^\Gamma$ is $\supp(x)=\{i\in\Gamma: x_i\neq 0\}$. If $\Lambda\subseteq\Gamma$, we identify in the natural way $K^\Lambda$ with the subgroup $\{x\in K^\Gamma:\supp(x)\subseteq \Lambda\}$ of $K^\Gamma$.

\smallskip
If $H$ is a subgroup of $K$, then $H^{\Gamma}$ is a $\sigma_{\lambda,K}$-invariant subgroup of $K^\Gamma$. Moreover, $\sigma_{\lambda,K}\restriction_{H^{\Gamma}}=\sigma_{\lambda,H}:H^\Gamma\to H^\Gamma$.

\begin{proposition}\label{composition}\label{G_F}
Let $\Gamma$ be a set, $\lambda,\mu:\Gamma\to\Gamma$ functions, $K$ a non-trivial abelian group, and consider the generalized shifts $\sigma_\lambda,\sigma_\mu:K^\Gamma\to K^\Gamma$. Then:
\begin{itemize}
    \item[(a)] $\sigma_\lambda\circ\sigma_\mu=\sigma_{\lambda\circ\mu}$ (hence $\sigma_\lambda^m=\sigma_{\lambda^m}$ for every $m\in\N$), and
    \item[(b)] $\lambda$ is injective (respectively, surjective) if and only if $\sigma_\lambda$ is surjective (respectively, injective). In particular, $\lambda$ is a bijection if and only if $\sigma_\lambda$ is an automorphism; in this case, $(\sigma_\mu)^{-1}=\sigma_{\mu^{-1}}$.
    \item[(c)] If $x\in K^\Gamma$, then $\supp(\sigma_\lambda^m(x))=\lambda^{-m}(\supp(x))$ for every $m\in\N$, and so
    \item[(d)] $\sigma_\lambda=\sigma_\mu$ if and only if $\lambda=\mu$.
\end{itemize}
\end{proposition}
\begin{proof}
For a proof of (a) and (b) see \cite{AADGH} and \cite{AKH}.

\smallskip
(c) If $y=\sigma_\lambda(x)$, then $i\in\supp(y)$ if and only if $y_i=x_{\lambda(i)}\neq0$, that is $\lambda(i)\in \supp(x)$; this is equivalent to $i\in\lambda^{-1}(\supp(x))$, and so $\supp(y)=\lambda^{-1}(\supp(x))$. Proceeding by induction it is possible to prove that $\supp(\sigma_\lambda^m(x))=\lambda^{-m}(\supp(x))$ for every $m\in\N$.

\smallskip
(d) If $\lambda=\mu$, the obviously $\sigma_\lambda=\sigma_\mu$. Assume then that $\sigma_\lambda=\sigma_\mu$. Let $j\in\Gamma$, $i=\lambda(j)$ and $x\in K^\Gamma$ be such that $\supp(x)=\{i\}$. By (c) and by hypothesis $\lambda^{-1}(i)=\supp(\sigma_\lambda(x))=\supp(\sigma_\mu(x))=\mu^{-1}(i)$. Then $j\in\lambda^{-1}(i)=\mu^{-1}(i)$ and in particular $\mu(j)=i=\lambda(j)$.
\end{proof}

Item (a) of next lemma gives a condition on $\lambda$ equivalent to the $\sigma_\lambda$-invariance for the ``rectangular" subgroups of $K^\Gamma$, while item (b) gives a sufficient condition for the algebraic entropy of a generalized shift to be infinite.

\begin{lemma}\label{invariance}
Let $\Gamma$ be a set, $\lambda:\Gamma\to \Gamma$ a function and $K$ a non-trivial finite abelian group. 
\begin{itemize}
\item[(a)] If $\Lambda\subseteq \Gamma$, then $K^\Lambda$ is $\sigma_\lambda$-invariant if and only if $\lambda^{-1}(\Lambda)\subseteq \Lambda$.
If $\Lambda\supseteq \lambda^{-1}(\Lambda)\cup\lambda(\Lambda)$, then $\sigma_\lambda\restriction_{K^\Lambda}=\sigma_{\lambda\restriction_\Lambda}$.
\item[(b)] If $\{\Lambda_i\}_{i\in I}$ is an infinite family of pairwise disjoint $\lambda^{-1}$-invariant subsets of $\Gamma$, and $\ent(\sigma_\lambda\restriction_{K^{\Lambda_i}})>0$ for every $i\in I$, then $\ent(\sigma_\lambda)=\infty$.
\end{itemize}
\end{lemma}
\begin{proof}
(a) Assume that $\lambda^{-1}(\Lambda)\subseteq \Lambda$ and let $x\in K^\Lambda$. Then $\supp(x)\subseteq \Lambda$. By Proposition \ref{G_F}(c) $\supp(\sigma_\lambda(x))=\lambda^{-1}(\supp(x))\subseteq \lambda^{-1}(\Lambda)$. Then $\sigma_\lambda(x)\in K^{\lambda^{-1}(\Lambda)}\subseteq K^\Lambda$. This shows that $\sigma_\lambda(K^\Lambda)\subseteq K^\Lambda$.
Suppose now that $\sigma_\lambda(K^\Lambda)\subseteq K^\Lambda$. Let $i\in\lambda^{-1}(\Lambda)$. Then $a=\lambda(i)\in \Lambda$ and, for $x\in K^\Gamma$ such that $\supp(x)=\{a\}$ we have $x\in K^\Lambda$. By Proposition \ref{G_F}(c) $\supp(\sigma_\lambda(x))=\lambda^{-1}(a)$ and by hypothesis $\sigma_\lambda(x)\in K^\Lambda$, so that $\lambda^{-1}(a)\subseteq \Lambda$; in particular $i\in \Lambda$, and hence $\lambda^{-1}(\Lambda)\subseteq \Lambda$.

If $\lambda^{-1}(\Lambda)\cup\lambda(\Lambda)\subseteq \Lambda$, then it is possible to consider both $\sigma_\lambda\restriction_{K^\Lambda}$ and $\sigma_{\lambda\restriction_\Lambda}$. It is clear that they coincide on $K^\Lambda$.

\smallskip
(b) By hypothesis $K^\Gamma$ contains a subgroup isomorphic to $\prod_{i\in I}K^{\Lambda_i}$. By (a) $K^{\Lambda_i}$ is $\sigma_\lambda$-invariant for every $i\in I$ and by Proposition \ref{properties}(d) $\ent(\sigma_\lambda)\geq\sum_{i\in I}\ent(\sigma_\lambda\restriction_{K^{\Lambda_i}})$. Since by hypothesis $\ent(\sigma_\lambda\restriction_{K^{\Lambda_i}})>0$, it follows that $\ent(\sigma_\lambda\restriction_{K^{\Lambda_i}})\geq\log 2$ (as $\ent(-)$ has as values either $\infty$ or $\log n$ for some $n\in\N_+$). Hence $\ent(\sigma_\lambda)=\infty$.
\end{proof}

The following proposition shows that the quasi-periodicity of a function $\lambda$ is equivalent to the quasi-periodicity of the generalized shift $\sigma_\lambda$.

\begin{proposition}\label{qp}
Let $\Gamma$ be a set, $\lambda:\Gamma\to\Gamma$ a function, $K$ a non-trivial abelian group and $\sigma_\lambda:K^\Gamma\to K^\Gamma$ the generalized shift. The the following conditions are equivalent:
\begin{itemize}
\item[(a)]$\lambda$ is quasi-periodic;
\item[(b)]$\sigma_\lambda$ is quasi-periodic.
\end{itemize}
In case $\lambda$ has finite fibers, also the following condition is equivalent to the previous ones:
\begin{itemize}
\item[(c)]$\sigma_\lambda\restriction_{\bigoplus_\Gamma K}$ is quasi-periodic.
\end{itemize}
\end{proposition}
\begin{proof}
(a)$\Leftrightarrow$(b) Assume that $\lambda$ is quasi-periodic, that is, there exist $n< m$ in $\N$ such that $\lambda^n=\lambda^{m}$. By Proposition \ref{composition}(a,d) this is equivalent to $\sigma_\lambda^n=\sigma_{\lambda^n}=\sigma_{\lambda^m}=\sigma_\lambda^m$, that is to say $\sigma_\lambda$ quasi-periodic.

\smallskip
Assume that $\lambda^{-1}(i)$ is finite for every $i\in\Gamma$.
Then (a)$\Leftrightarrow$(c) can be proved exactly as (a)$\Leftrightarrow$(b), observing that for some $n<m$ in $\N$, as in Proposition \ref{composition}(d), $\lambda^n=\lambda^m$ if and only if $\sigma_{\lambda^n}=\sigma_{\lambda^m}$.
\end{proof}

\section{Strings, infinite orbits, ladders and bounded functions}\label{strings}

The main goal of this section is to characterize the bounded functions, proving Theorem \ref{bounded}. To this end we need the notions in the next Definition \ref{sol}.

\smallskip
First we fix some notations. By $\N_+$ we denote the set of positive integers.
Let $X$ be a set and $f:X\to X$ a function. We say that a point $x\in X$ is \emph{periodic} for $f$ if there exists $n\in\N_+$, such that $f^n(x)=x$. The \emph{period} of a periodic point $x\in X$ of $f$ is the minimum positive integer $n$ such that $f^n(x)=x$ (i.e., $n$ is the length of the orbit of $x$). Let $\Per(f)$ be the set of all periodic points and for $n\in\N_+$ let $\Per_n(f)$ be the set of all periodic points of period at most $n$ of $f$ in $X$. The function $f$ is \emph{periodic} if there exists $n\in\N_+$ such that $f^n=id_X$, that is, $\Per(f)=\Per_m(f)$ for some $m\in\N_+$.

\begin{deff}\label{sol}
Let $\Gamma$ be a set and $\lambda:\Gamma\to\Gamma$ a function.
\begin{itemize}
\item[(a)] A \emph{string} of $\lambda$ (in $\Gamma$) is an infinite sequence $S=\{s_t\}_{t\in\N}$ of pairwise distinct elements of $\Gamma$, such that $\lambda(s_t)=s_{t-1}$ for every $t\in\N_+$.
\item[(b)] An \emph{infinite orbit} of $\lambda$ (in $\Gamma$) is an infinite sequence $A=\{a_t\}_{t\in\N}$ of pairwise distinct elements of $\Gamma$, such that $\lambda(a_t)=a_{t+1}$ for every $t\in\N$.
\item[(c)] A \emph{ladder} of $\lambda$ (in $\Gamma$) is a subset $L$ of $\Gamma$ such that $L=\bigcup_{m\in\N}L_m$ is a disjoint union of non-empty finite subsets $L_m=\{l_{m,0},l_{m,1},\ldots,l_{m,b_m}\}$ of $\Gamma\setminus\Per(\lambda)$, with
\begin{itemize}
\item[(i)] $\lambda(l_{m,k})=l_{m,k-1}$ for every $k\in\{1,\ldots,b_m\}$ and $\lambda^{-1}(l_{m,b_m})=\emptyset$.
\item[(ii)] Moreover, $b_s<b_t$ in case $s<t$ in $\N$.
\end{itemize}
\item[(d)] A \emph{periodic ladder} of $\lambda$ (in $\Gamma$) is a subset $P$ of $\Gamma$ such that $P$ has a partition $P=\bigcup_{n\in\N_+} P_n$, where $P_n$ is finite, $|P_n|\geq n$ and $\lambda\restriction_{P_n}:P_n\to P_n$ is a cycle of length $|P_n|$ for every $n\in\N_+$.
\end{itemize}
\end{deff}

Note that by the definition of ladder, $b_m\geq m$, and so $|L_m|>m$, for every $m\in\N$.

\medskip
The following diagram represents a string $S=\{s_t\}_{t\in\N}$, an infinite orbit $A=\{a_t\}_{t\in\N}$ of an element $a_0$, a ladder $L=\bigcup_{m\in\N}\{l_{m,0},l_{m,1},\ldots,l_{m,m}\}$ (i.e., $b_m=m$ for every $m\in\N$), and a periodic ladder $P=\bigcup_{n\in\N_+}P_n$ in case $P_n=\{p_{n,1},\ldots,p_{n,n}\}$ (in particular, $|P_n|=n$) for every $n\in\N_+$.
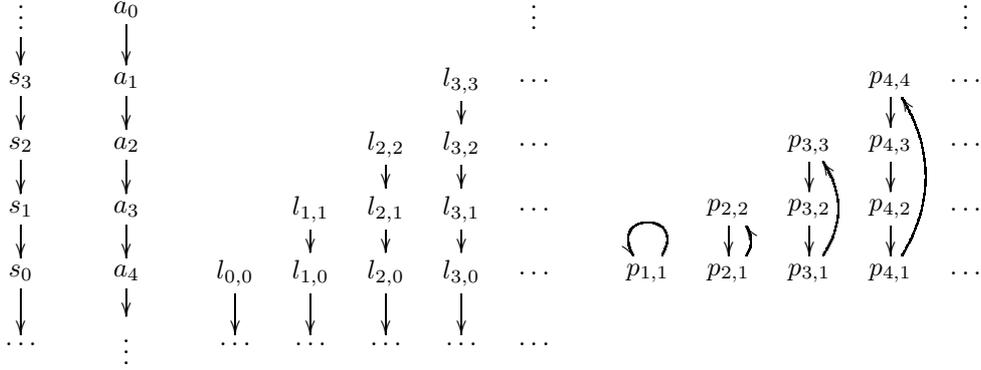
\begin{equation*}
\xymatrix@-1.3pc{
 \vdots\ar[d] && a_0\ar[d] &&                  &                     &                    &                     & \vdots          &&& & &  & \vdots \\
 s_3\ar[d]    && a_1\ar[d] &&                   &                     &                    &  l_{3,3}\ar[d] & \ldots           &&& & & p_{4,4}\ar[d] & \ldots \\
 s_2\ar[d]    && a_2\ar[d] &&                   &                     & l_{2,2}\ar[d] & l_{3,2}\ar[d] & \ldots           &&&       & p_{3,3}\ar[d] & p_{4,3}\ar[d] & \ldots \\
 s_1\ar[d]    && a_3\ar[d] &&                  & l_{1,1}\ar[d]  & l_{2,1}\ar[d] & l_{3,1}\ar[d] & \ldots           &&& p_{2,2}\ar[d] & p_{3,2}\ar[d] & p_{4,2}\ar[d] & \ldots \\
 s_0\ar[d]    && a_4\ar[d] &&l_{0,0}\ar[d] & l_{1,0}\ar[d] & l_{2,0}\ar[d] & l_{3,0}\ar[d] & \ldots           &&  p_{1,1}\ar@(ur,ul)[] & p_{2,1}\ar@/_0.7pc/[u] & p_{3,1}\ar@/_0.9pc/[uu] & p_{4,1}\ar@/_1.1pc/[uuu] & \ldots\\
 \ldots &  &      \vdots & &       \ldots & \ldots & \ldots & \ldots & \ldots
}
\bigskip
\end{equation*}

Following \cite{AADGH}, a string $S=\{s_t\}_{t\in\N}$ of $\lambda$ in $\Gamma$ is \emph{acyclic} if $\lambda^n(s_0)\not\in S$ for every $n\in\N_+$. Then an acyclic string is an ascending chain in $(\Gamma,\leq_\lambda)$.
It is easy to prove that every string contains an acyclic string.

An infinite orbit $A$ of $\lambda$ is a totally ordered subset of $(\Gamma,\leq_\lambda)$, as well as each $L_m$ in case $L=\bigcup_{m\in\N}L_m$ is a ladder of $\lambda$, since $A$ and the $L_m$'s meet trivially $\Per(\lambda)$. More precisely, $A$ can be viewed also as an infinite descending chain, and each $L_m$ is a finite chain such that its top element is a maximal element in $(\Gamma,\leq_\lambda)$ (so a ladder is disjoint union of finite chains of strictly increasing length and each finite chain ends with a maximal element of $(\Gamma,\leq_\lambda)$). Therefore a surjective $\lambda$ has no ladder.

Finally note that the existence of a periodic ladder $P$ of $\lambda$ in $\Gamma$ is equivalent to the existence of periodic points of arbitrarily large order, that is, to the existence of periodic orbits of arbitrarily large length, i.e., $\Per(\lambda)\supsetneq\Per_n(\lambda)$ for every $n\in\N_+$.

\medskip
We introduce now cardinal invariants that measure respectively the number of pairwise disjoint strings, pairwise disjoint infinite orbits and pairwise disjoint ladders of a function. The first of them was already introduced in \cite{AADGH}.

\begin{deff}\label{sol2}
Let $\Gamma$ be a set and $\lambda:\Gamma\to\Gamma$ a function. Then let:
\begin{itemize}
\item[(a)] $s(\lambda)=\sup\{|\mathcal F|: \mathcal F\ \text{is a family of pairwise disjoint strings in $\Gamma$}\}$;
\item[(b)] $o(\lambda)=\sup\{|\mathcal F|: \mathcal F\ \text{is a family of pairwise disjoint infinite orbits in $\Gamma$}\}$;
\item[(c)] $l(\lambda)=\sup\{|\mathcal F|: \mathcal F\ \text{is a family of pairwise disjoint ladders in $\Gamma$}\}$;
\item[(d)] $p(\lambda)=\sup\{|\mathcal F|: \mathcal F\ \text{is a family of pairwise disjoint periodic ladders in $\Gamma$}\}$.
\end{itemize}
\end{deff}

The existence of a ladder of $\lambda$ in $\Gamma$ is equivalent to the existence of infinitely many pairwise disjoint ladders of $\lambda$ in $\Gamma$; in other words $l(\lambda)>0$ implies $l(\lambda)\geq\omega$. Analogously, $p(\lambda)>0$ yields $p(\lambda)\geq\omega$.

\medskip
The next result gives five equivalent characterizations of a bounded function. In particular, it shows that bounded is the same as quasi-periodic for a function; that a function is bounded if and only if it has no strings, no infinite orbits, no ladders; and the period of the periodic points is bounded by a fixed $N\in\N_+$.

\begin{theorem}\label{bounded}
Let $\Gamma$ be a set and $\lambda:\Gamma\to\Gamma$ a function. The following conditions are equivalent:
\begin{itemize}
\item[(a)]$\lambda$ is bounded;
\item[(b)]$s(\lambda)=o(\lambda)=l(\lambda)=p(\lambda)=0$;
\item[(c)]there exists $N\in\N_+$ such that $\lambda^{-N}(\Gamma\setminus\Per(\lambda))=\emptyset$ and $\Per(\lambda)=\Per_N(\lambda)$;
\item[(d)]there exists $N\in\N_+$ such that $\lambda^{N}(\Gamma)=\Per(\lambda)$ and $\Per(\lambda)=\Per_N(\lambda)$;
\item[(e)]$\lambda$ is quasi-periodic.
\end{itemize}
\end{theorem}
\begin{proof}
(a)$\Rightarrow$(b) If either $s(\lambda)>0$, or $o(\lambda)>0$, then there exists either a string or an infinite orbit of $\lambda$ in $\Gamma$, both of which are infinite totally preordered subsets of $(\Gamma,\leq_\lambda)$, and so $\lambda$ is not bounded. If $l(\lambda)>0$, then there exists a ladder $L=\bigcup_{m\in\N}L_m$ of $\lambda$ in $\Gamma$; in particular, for every $m\in\N$ the set $L_m$ is a totally ordered subset of $(\Gamma,\leq_\lambda)$ of size $> m$, and hence $\lambda$ is not bounded. If $p(\lambda)>0$, then there exists a periodic ladder $P=\bigcup_{n\in\N_+}P_n$; each $P_n$ is a totally preordered subset of $(\Gamma,\leq_\lambda)$ of size $\geq n$, and so $\lambda$ is not bounded.

\smallskip
(b)$\Rightarrow$(c) Suppose that for every $n\in\N_+$ there exists $i_n\in\Gamma\setminus\Per(\lambda)$ such that $\lambda^{-n}(i_n)$ is not empty. 
Since $p(\lambda)=0$ is equivalent to $\Per(\lambda)=\Per_N(\lambda)$ for some $N\in\N_+$, and since $s(\lambda)=o(\lambda)=0$, we have to verify that $l(\lambda)>0$. To this end we construct a ladder of $\lambda$ in $\Gamma$.

First note that, given $i\in\Gamma\setminus\Per(\lambda)$, since $s(\lambda)=o(\lambda)=0$, there exist $n,m\in\N_+$ such that $\lambda^n(i)\in\Per(\lambda)$ and $\lambda^{-m}(i)=\emptyset$. So we can suppose without loss of generality (i.e., taking $\lambda^{n-1}(i)$ instead of $i$) that $\lambda(i)\in\Per(\lambda)$ and $\lambda^{-m}(i)=\emptyset$ for some $m\in\N_+$.

So let $l_0\in\Gamma\setminus\Per(\lambda)$ be such that $\lambda(l_0)\in\Per(\lambda)$ and $\lambda^{-b_0-1}(l_0)=\emptyset$ where $b_0$ is the minimum natural number with this property. Then pick $l_{0,b_0}\in\lambda^{-b_0}(l_0)$, let $l_{0,b_0-k}=\lambda^k(l_{0,b_0})$ for every $k\in\{0,\ldots,b_0\}$ (in particular $l_{0,0}=l_0$) and define $$L_0=\{l_{0,0},l_{0,1},\ldots,l_{0,b_0}\}.$$ By our assumption there exists $l_1\in\Gamma\setminus\Per(\lambda)$ such that $\lambda(l_1)\in\Per(\lambda)$ and $\lambda^{-b_0-1}(l_1)\neq\emptyset$. Let $b_1$ be the minimum natural number such that $\lambda^{-b_1-1}(l_1)=\emptyset$; in particular $b_1>b_0$. Pick $l_{1,b_1}\in\lambda^{-b_1}(l_1)$, let $l_{1,b_1-k}=\lambda^k(l_{1,b_1})$ for every $k\in\{0,\ldots,b_1\}$ (in particular $l_{1,0}=l_1$) and define $$L_1=\{l_{1,0},l_{1,1},\ldots,l_{1,b_1}\}.$$  Proceeding by induction in this way, for every $m\in\N$ we have $l_m\in\Gamma\setminus\Per(\lambda)$ such that $\lambda(l_m)\in\Per(\lambda)$ and there exists a minimum natural number $b_m$ such that $\lambda^{-b_m-1}(l_m)=\emptyset$, and $b_m>b_n$ for every $n<m$ in $\N$. Moreover, $$L_m=\{l_{m,0},l_{m,1},\ldots,l_{m,b_m}\},$$ where $l_{m,b_m-k}=\lambda^k(l_{m,b_m})$ for every $k\in\{0,\ldots,b_m\}$ and in particular $l_{m,0}=l_m$.
 
By construction, for each $m\in\N$ we have $L_m\subseteq\Gamma\setminus\Per(\lambda)$, $\lambda^{-1}(l_{m,b_m})=\emptyset$ and $b_s<b_t$ for every $s<t$ in $\N$. Let $L=\bigcup_{m\in\N}L_m$. To prove that this is a ladder it remains to verify that $L_s\cap L_t=\emptyset$ in case $s\neq t$ in $\N$. So let $s\leq t$ in $\N$ and suppose that $L_s\cap L_t$ is not empty. This means that there exist $l_{s,v}\in L_s$ and $l_{t,w}\in L_t$ such that $l_{s,v}=l_{t,w}$. If $v\leq w$, this implies $l_{s,0}=l_{t,w-v}$; since $\lambda(l_{t,w-v})=\lambda(l_{s,0})\in\Per(\lambda)$, this equality is possible only if $w-v=0$, that is $v=w$. The same conclusion holds assuming $w\leq v$. Then $l_{s,0}=l_{t,0}$. Since $\lambda^{-b_s-1}(l_{t,0})=\lambda^{-b_s-1}(l_{s,0})=\emptyset$, it follows that $t\leq s$. Hence $s=t$, and $L_s=L_t$.

\smallskip
(c)$\Rightarrow$(a) Assume that there exists $N\in\N_+$ such that $\lambda^{-N}(\Gamma\setminus\Per(\lambda))=\emptyset$ and $\Per(\lambda)=\Per_N(\lambda)$. Let $I$ be a totally preordered subset of $(\Gamma,\leq_\lambda)$. Since $I\setminus\Per(\lambda)$ is totally ordered and $\lambda^{-N}(I\setminus\Per(\lambda))=\emptyset$, so $I\setminus\Per(\lambda)$ has size at most $N$; moreover, $I\cap\Per(\lambda)$, being totally preordered, is contained in an orbit of size at most $N$. It follows that $|I|\leq 2 N$. This proves that $\lambda$ is bounded.

\smallskip
(c)$\Leftrightarrow$(d)$\Leftrightarrow$(e) are obvious.



\end{proof}

\section{Quasi-periodicity coincides with local quasi-pe\-rio\-di\-ci\-ty for the generalized shifts}\label{qp=lqp}

In this section we prove in Theorem \ref{pre-mt} the equivalence of local quasi-periodicity and quasi-periodicity for a generalized shift $\sigma_\lambda$, and that these conditions are equivalent also to the boundedness of $\lambda$ (i.e., the quasi-periodicity of $\lambda$ in view of Theorem \ref{bounded}).
The non-trivial part in this proof is to find a non-quasi-periodic point of $\sigma_\lambda$ under the assumption that $\lambda$ is not bounded.

\medskip
For a set $\Gamma$ and an abelian group $K$, the \emph{diagonal} subgroup $\Delta K^\Gamma$ of $K^\Gamma$ is $\Delta K^\Gamma=\{x=(x_i)_{i\in\Gamma}:\text{for some}\ a\in K,\ x_i=a\ \text{for every}\ i\in\Gamma\}$.

\begin{lemma}\label{uff}
Let $\Gamma$ be a set, $\lambda:\Gamma\to \Gamma$ a function, $K$ a non-trivial abelian group, and $\sigma_\lambda:K^\Gamma\to K^\Gamma$ the generalized shift.
\begin{itemize}
\item[(a)]If $\lambda$ has a ladder $L$ in $\Gamma$, then there exists $x\in K^L$ which is not quasi-periodic for $\sigma_\lambda$. So $\ent(\sigma_\lambda)>0$.
\item[(b)]if $\lambda$ has a periodic ladder $P$ in $\Gamma$, then there exists $x\in K^P$ which is not quasi-periodic for $\sigma_\lambda$. So $\ent(\sigma_\lambda)>0$.
\end{itemize}
\end{lemma}
\begin{proof}
(a) Let $L=\bigcup_{m\in\N}L_m\subseteq \Gamma\setminus\Per(\lambda)$, where each $L_m=\{l_{m,0},\ldots,l_{m, b_m}\}$. Consider where $B=\{l_{m,0}:m\in \N\}$ and let $x$ be a non-zero element of $\Delta K^B$. We show that $x$ is not quasi-periodic for $\sigma_\lambda$. To this end let $s<t$ in $\N$.
By Proposition \ref{G_F}(c) $\supp(\sigma_\lambda^s(x))={\lambda^{-s}(B)}$ and $\supp(\sigma_\lambda^t(x))={\lambda^{-t}(B)}$. Then $\supp(\sigma_\lambda^s(x))\cap L=\{l_{m,s}:m\in\N,m\geq s\}$ and $\supp(\sigma_\lambda^t(x))\cap L=\{l_{m,t}:m\in\N,m\geq t\}$. By the definition of ladder the latter two sets have trivial intersection as $s<t$. So, since $\sigma_\lambda^s(x)\in\Delta K^{\supp(\sigma_\lambda^s(x))}$ and $\sigma_\lambda^t(x)\in\Delta K^{\supp(\sigma_\lambda^t(x))}$, it follows that $\sigma_\lambda^s(x)\neq\sigma_\lambda^t(x)$. This proves that $x$ is not quasi-periodic.

\smallskip
(b) Let $P=\bigcup_{n\in\N_+}P_n$. By definition, for every $n\in\N_+$ we have $P_n\supseteq\lambda^{-1}(P_n)\cup\lambda(P_n)$. By Lemma \ref{invariance}(a) $\sigma_\lambda\restriction_{K^{P_n}}=\sigma_{\lambda_n}$, where $\lambda_n=\lambda\restriction_{P_n}$ for every $n\in\N_+$. Moreover, $K^P\cong\prod_{n\in\N_+}K^{P_n}$, and so $\sigma_\lambda\restriction_{K^P}=(\sigma_{\lambda_n})_{n\in\N_+}$.

For every $n\in \N_+$ let $x_n=(x_{n,s})_{s\in P_n}\in K^{P_n}$ be such that $x_{n,s_n}\neq 0$ for one and only one $s_n\in P_n$, that is, $\supp(x_n)=\{s_n\}$.
Let $x=(x_n)_{n\in \N_+}\in \prod_{n\in \N_+}K^{P_n}$. We show that $x$ is not quasi-periodic for $\sigma_\lambda$. 
Let $s<t$ in $\N$. We have to verify that $\sigma_\lambda^s(x)\neq\sigma_\lambda^t(x)$. Since $\sigma_\lambda^s(x)=(\sigma_{\lambda_n}^s(x_n))_{n\in\N_+}$ and $\sigma_\lambda^t(x)=(\sigma_{\lambda_n}^t(x_n))_{n\in\N_+}$, it suffices to show that there exists $n\in\N_+$ such that $\sigma_{\lambda_n}^s(x_n)\neq\sigma_{\lambda_n}^t(x_n)$. Take for example $n\in\N_+$ such that $|P_n|>t$. Then $\lambda_n^{-s}(s_n)\neq \lambda_n^{-t}(s_n)$ in view of the hypothesis that $\lambda_n$ is a cycle of length $|P_n|>t>s$, and by Proposition \ref{G_F}(c) $\supp(\sigma_{\lambda_n}^s(x_n))=\{\lambda_n^{-s}(s_n)\}\neq\{\lambda^{-t}(s_n)\}=\supp(\sigma_{\lambda_n}^t(s_n))$. Hence $\sigma_\lambda^s(x)\neq\sigma_\lambda^t(x)$.

\smallskip
In both (a) and (b) the existence of a non-quasi-periodic point of $\sigma_\lambda$ implies $\ent(\sigma_\lambda)>0$ in view of Proposition \ref{properties}(a).
\end{proof}

\begin{theorem}\label{pre-mt}
Let $\Gamma$ be a set, $\lambda:\Gamma\to \Gamma$ a function, $K$ a non-trivial abelian group, and $\sigma_\lambda:K^\Gamma\to K^\Gamma$ the generalized shift. The following conditions are equivalent:
\begin{itemize}
\item[(a)] $\lambda$ is bounded;
\item[(b)] $\sigma_\lambda$ is quasi-periodic;
\item[(c)] $\sigma_\lambda$ is locally quasi-periodic.
\end{itemize}
\end{theorem}
\begin{proof}
(a)$\Leftrightarrow$(b) By Theorem \ref{bounded} $\lambda$ is bounded if and only if $\lambda$ is quasi-periodic. Then apply Proposition \ref{qp} to conclude that $\lambda$ quasi-periodic is equivalent to $\sigma_\lambda$ quasi-periodic.

\smallskip
(b)$\Rightarrow$(c) is obvious.

\smallskip
(c)$\Rightarrow$(a) We verify that in case $\lambda$ is not bounded, then $\sigma_\lambda$ is not locally quasi-periodic, that is, there exists $x\in K^\Gamma$ which is not quasi-periodic. By Theorem \ref{bounded} $\lambda$ non-bounded means that one of $s(\lambda)$, $o(\lambda)$, $l(\lambda)$, $p(\lambda)$ is non-zero.

Let $$N_1=\{n!:n\in\N_+\}\subseteq\N$$ and for every $k\in\N$ let 
$$N_1+k=\{n+k:n\in N_1\}\ \text{and}\ N_1-k=\{n-k:n\in N_1,n>k\}.$$

Suppose that $s(\lambda)>0$. Then there exists a string $S=\{s_t\}_{t\in\N}$ of $\lambda$ in $\Gamma$.
For $k\in\N$ define $$S_{1,k}=\{s_n:n\in N_1+k\}\subseteq S.$$
Let $x$ be a non-zero element of $\Delta K^{S_{1,0}}$. We verify that $x$ is not quasi-periodic for $\sigma_\lambda$. To this aim, let $s<t$ in $\N$. By Proposition \ref{G_F}(c) $\supp(\sigma_\lambda^s(x))=\lambda^{-s}(S_{1,0})$ and $\supp(\sigma_\lambda^t(x))=\lambda^{-t}(S_{1,0})$. Then $\supp(\sigma_\lambda^s(x))\cap S=S_{1,s}$ and $\supp(\sigma_\lambda^t(x))\cap S=S_{1,t}$. In particular, $S_{1,s}\neq S_{1,t}$ because $N_1+s\neq N_1+t$, and so $\supp(\sigma_\lambda^s(x))\neq \supp(\sigma_\lambda^t(x))$. Since $\sigma_\lambda^s(x)$ and $\sigma_\lambda^t(x)$ are elements respectively of $\Delta K^{\supp(\sigma_\lambda^s(x))}$ and $\Delta K^{\supp(\sigma_\lambda^t(x))}$, it follows that $\sigma_\lambda^s(x)\neq\sigma_\lambda^t(x)$.

Suppose that $o(\lambda)>0$. Then there exists an infinite orbit $A=\{a_t\}_{t\in\N}$ of $\lambda$ in $\Gamma$. 
For $k\in\N$ define $$A_{1,k}=\{a_n:n\in N_1-k\}\subseteq A.$$
Let $x$ be a non-zero element of $\Delta K^{A_{1,0}}$. We verify that $x$ is not quasi-periodic for $\sigma_\lambda$. To this aim, let $s<t$ in $\N$. By Proposition \ref{G_F}(c) $\supp(\sigma_\lambda^s(x))=\lambda^{-s}(A_{1,0})$ and $\supp(\sigma_\lambda^t(x))=\lambda^{-t}(A_{1,0})$. Then $\supp(\sigma_\lambda^s(x))\cap A=A_{1,s}$ and $\supp(\sigma_\lambda^t(x))\cap A=A_{1,t}$. In particular, $A_{1,s}\neq A_{1,t}$ because $N_1-s\neq N_1-t$, and so $\supp(\sigma_\lambda^s(x))\neq \supp(\sigma_\lambda^t(x))$. Since $\sigma_\lambda^s(x)$ and $\sigma_\lambda^t(x)$ are elements respectively of $\Delta K^{\supp(\sigma_\lambda^s(x))}$ and $\Delta K^{\supp(\sigma_\lambda^t(x))}$, it follows that $\sigma_\lambda^s(x)\neq\sigma_\lambda^t(x)$.

If $l(\lambda)>0$, apply Lemma \ref{uff}(a), and if $p(\lambda)>0$, apply Lemma \ref{uff}(b).
\end{proof}

According to \cite{DGSZ}, a function $f:X\to X$ is \emph{strongly recurrent} if it is locally periodic.
In \cite{DGSZ} an analogue of the Poincar\'e -- Birkhoff recurrence theorem of ergodic theory was proved: 
\begin{quote}
\emph{For $\phi$ a monomorphism of a torsion abelian group,  $\phi$ is locally periodic (i.e., strongly recurrent) if and only if $\ent(\phi)=0$.}
\end{quote}
Similarly to the situation in Theorem \ref{mt}, for injective generalized shifts $\sigma_\lambda$ the ``local" condition becomes ``global":

\begin{corollary}\label{recurrence}
Let $\Gamma$ be a set, $\lambda:\Gamma\to\Gamma$ a function, $K$ a non-trivial finite abelian group and $\sigma_\lambda:K^\Gamma\to K^\Gamma$ an injective generalized shift. Then the following conditions are equivalent:
\begin{itemize}
\item[(a)]$\sigma_\lambda$ is locally periodic (i.e., strongly recurrent);
\item[(b)]$\sigma_\lambda$ is periodic;
\item[(c)]$\ent(\sigma_\lambda)=0$.
\end{itemize}
\end{corollary}
\begin{proof}
(a)$\Leftrightarrow$(c) was proved in \cite{DGSZ}, and (b)$\Rightarrow$(a) is clear.

\smallskip
(c)$\Rightarrow$(b) Assume that $\ent(\sigma_\lambda)=0$. By Proposition \ref{properties}(a) $\sigma_\lambda$ is locally quasi-periodic and by Theorem \ref{pre-mt} $\sigma_\lambda$ is quasi-periodic. Since it is injective, $\sigma_\lambda$ is periodic.
\end{proof}

\section{Independent subgroups of $K^\N$}\label{disjoint}

Let us give the following definition, which will help in explaining the content of this section.

\begin{deff}
Let $G$ be an abelian group. A family $\{H_i:i\in I\}$ of subgroups of $G$ is \emph{independent} if for any finite subset $J=\{j_1,\ldots,j_n\}$ of $I$ and any $j_0\in I\setminus J$ then $H_{j_0}\cap(H_{j_1}+\ldots+ H_{j_n})=\{0\}$.
\end{deff}

In particular, the $H_i$'s in this definition are pairwise with trivial intersection. Observe that a family $\{H_n:n\in\N\}$ of subgroups of $G$ is independent if and only if $H_{n+1}\cap (H_0+\ldots+H_n)=\{0\}$ for every $n\in\N$.

\medskip
The subsets $N_1+k$ and $N_1-k$ of $\N$ in the proof of Theorem \ref{pre-mt} help in finding a non-quasi-periodic point of $\sigma_\lambda$ when $\lambda$ admits either a string or an infinite orbit. By Proposition \ref{properties}(a) this is equivalent to say that $\ent(\sigma_\lambda)>0$. But to prove Theorem \ref{mt} we have to show that this entropy is infinite and so we have to improve the use of the subsets $N_1+k$ and $N_1-k$ of $\N$. 

With this aim, we consider in this section similar subsets of $\N$ defined through the use of the factorial of natural numbers. The properties of these subsets help in finding in Lemmas \ref{pseudo-eq1-lemma} and \ref{pseudo-eq2-lemma} specific independent families of finite subgroups of $K^\N$. These subgroups are ``sufficiently many" with respect to the calculation of the algebraic entropy of a generalized shift $\sigma_\lambda$ and are useful to prove Lemma \ref{asl>0->ent=infty}, in which we see that the algebraic entropy of $\sigma_\lambda$ is infinite in case $\lambda$ admits either a string or an infinite orbit.

\medskip
For every $m,n\in\N_+$, let 
\begin{center}
$n!^{(m)}=n\underbrace{!\ldots!}_m$ and $N_m=\{n!^{(m)}:n\in\N_+\}$.
\end{center}
These subsets of $\N$ form a (rapidly) strictly decreasing sequence $$N_1\supset N_2\supset\ldots\supset N_m\supset N_{m+1}\supset\ldots;$$ indeed, $N_m\setminus N_{m+1}$ is infinite for every $m\in\N_+$. For $m\in\N_+$ and $k\in\N$ let 
\begin{equation}\label{N_m}
N_m+k=\{n+k:n\in N_m\}\ \text{and}\ N_m-k=\{n-k:n\in N_m,\ n>k\}.
\end{equation}
We collect here some useful properties of these subsets $N_m$ of $\N$.

\begin{lemma}\label{subsetneq}
For every $m,k\in\N_+$,
\begin{itemize}
\item[(a)] $(N_m+k)\setminus (N_{m+1}+k)\not\subseteq N_1\cup (N_1+1)\cup\ldots\cup (N_1+(k-1))$, and
\item[(b)] $(N_m-k)\setminus (N_{m+1}-k)\not\subseteq N_1\cup (N_1-1)\cup\ldots\cup (N_1-(k-1))$.
\end{itemize}
\end{lemma}
\begin{proof}
(a) Let $m,k\in\N_+$. We have to prove that there exists $n_0\in\N_+\setminus N_1$ such that $n_0!^{(m)}+k\neq n!+h$ (i.e., $n_0!^{(m)}+(k-h)\neq n!$) for every $n\in\N_+$ and $h\in\{0,\ldots,k-1\}$. 

Pick $n_0\in\N_+$ such that $k< M\cdot M!$, where $M= n_0!^{(m-1)}$ and so $M!=n_0!^{(m)}$
 (it suffices for example that $n_0>k$). In particular $n_0>1$ and for every $h\in\{0,\ldots,k-1\}$
$$
k-h<M\cdot M!.
$$
Consequently,
$$
M!+(k-h)<M!+M\cdot M!=(M+1)!.
$$
Then for every $h\in\{0,\ldots,k-1\}$, 
\begin{equation}\label{!m-a}
M!<M!+(k-h)<(M+1)!.
\end{equation}
Since $(M+1)!$ is the smallest factorial bigger than $M!$, it follows that $M!+(k-h)\neq n!$ for every $n\in\N$. 

If $n_0=n_1!$ for some $n_1\in\N_+$ (i.e., $n_0\in N_1$), then take $n_0+1$ and $M= (n_0+1)!^{(m-1)}$, which satisfies the same condition \eqref{!m-a} but $n_0+1\not\in N_1$.

\smallskip
(b) Let $m,k\in\N_+$. We have to prove that there exists $n_0\in\N_+\setminus N_1$ such that $n_0!^{(m)}-k\neq n!-h$ (i.e., $n_0!^{(m)}-(k-h)\neq n!$) for every $n\in\N_+$ and $h\in\{0,\ldots,k-1\}$.

Pick $n_0\in\N_+$ such that $k<(M-1)\cdot (M-1)!$, where as before $M= n_0!^{(m-1)}$ and so $M!=n_0!^{(m)}$ (it suffices for example that $n_0>k$).
In particular $n_0>1$ and for every $h\in\{0,\ldots,k-1\}$ $k-h<(M-1)\cdot (M-1)!$, that is,
$$-(k-h)>(M-1)\cdot (M-1)!.$$
Consequently,
$$
(M-1)!=M!-(M-1)\cdot (M-1)!<M!-(k-h).
$$
Then, for every $h\in\{0,\ldots,k-1\}$,
\begin{equation}\label{!m-b}
(M-1)!<M!-(k-h)<M!.
\end{equation}
Since $(M-1)!$ is the biggest factorial smaller than $M!$, it follows that $M!-(k-h)\neq n!$ for every $n\in\N$.

If $n_0=n_1!$ for some $n_1\in\N_+$ (i.e., $n_0\in N_1$), then take $n_0+1$ and $M= (n_0+1)!^{(m-1)}$, which satisfies the same condition \eqref{!m-b} and $n_0+1\not\in N_1$.
\end{proof}

In particular, it follows from this lemma that for every $m,k\in\N_+$,
\begin{align*}
& N_m+k\not\subseteq N_1\cup (N_1+1)\cup\ldots\cup (N_1+(k-1)),\ \text{and} \\
& N_m-k\not\subseteq N_1\cup (N_1-1)\cup\ldots\cup (N_1-(k-1)).
\end{align*}

\begin{remark}\label{supp(x)}
Consider the group $K^{\N}$, where $K$ is a non-trivial finite abelian group. Let $t\in\N_+$ and $k\in\Z$. If $x\in \Delta K^{N_1+k}+\ldots+\Delta K^{N_t+k}$ then $\supp(x)=Q_1\dot{\cup}\ldots\dot{\cup}Q_t,$ where
\begin{align*} 
Q_1&=\begin{cases} \text{either} & (N_1+k)\setminus (N_2+k)\\ \text{or} & \emptyset \end{cases}, \\ 
 & \vdots \\
 Q_{t-1}&=\begin{cases} \text{either} & (N_{t-1}+k)\setminus (N_t+k)\\ \text{or} & \emptyset \end{cases}, \\
 Q_t&=\begin{cases} \text{either} & N_t+k\\ \text{or} & \emptyset \end{cases}.
\end{align*}
In particular, if $\supp(x)\cap (N_t+k)$ is not empty, then $Q_t$ is not empty. Therefore $Q_t=N_t+k$, and hence $\supp(x)\supseteq N_t+k$.
\end{remark}

\begin{lemma}\label{>N_s+1->>N_s}
Let $t\in\N_+$, $k\in\Z$ and let $x\in\Delta K^{N_1+k}+\ldots+\Delta K^{N_t+k}$. If $\supp(x)\subsetneq N_t+k$, then $x=0$.
\end{lemma}
\begin{proof}
By Remark \ref{supp(x)}, if $\supp(x)\cap (N_t+k)\neq\emptyset$, it follows that $\supp(x)\supseteq N_t+k$. Then $\supp(x)=\supp(x)\cap (N_t+k)=\emptyset$, that is, $x=0$.
\end{proof}

The following result shows that for every $k\in\N$ the family $\{\Delta K^{N_t+k}:t\in\N_+\}$ of finite subgroups of $K^\N$ is independent.

\begin{lemma}\label{pseudo-eq1-lemma}
Consider the group $K^{\N}$, where $K$ is a non-trivial finite abelian group. If $k\in\N$ is fixed, then for every $t\in\N_+$,
\begin{itemize}
\item[(a)] $\Delta K^{N_1+k}+\ldots+\Delta K^{N_t+k}=\Delta K^{N_1+k}\oplus\ldots\oplus\Delta K^{N_t+k}$;
\item[(b)] $\Delta K^{N_1-k}+\ldots+\Delta K^{N_t-k}=\Delta K^{N_1-k}\oplus\ldots\oplus\Delta K^{N_t-k}$.
\end{itemize}
\end{lemma}
\begin{proof}
(a) We proceed by induction. Let $t=2$. Since $N_1+k\supsetneq N_2+k$, it follows that $\Delta K^{{N_1+k}}\cap \Delta K^{{N_2+k}}=\{0\}$. 
Assume now that for $t\geq2$, $\Delta K^{N_1+k}+\ldots+\Delta K^{N_t+k}=\Delta K^{N_1+k}\oplus\ldots\oplus\Delta K^{N_{t}+k}$; we prove that $$(\Delta K^{N_1+k}\oplus\ldots\oplus\Delta K^{N_{t}+k})\cap \Delta K^{N_{t+1}+k}=\{0\}.$$ To this end let $x\in\Delta K^{N_{t+1}+k}$. Then $\supp(x)$ is either empty or $N_{t+1}+k$. 
Since $N_{t+1}+k\subsetneq N_{t}+k$, and in particular $\supp(x)\subsetneq N_{t}+k$, by Lemma \ref{>N_s+1->>N_s} $x\in\Delta K^{N_1+k}\oplus\ldots\oplus\Delta K^{N_{t}+k}$ yields $x=0$. This concludes the proof.

\smallskip
(b) is analogous to (a).
\end{proof}

\begin{lemma}\label{pseudo-eq2-lemma}
Consider $K^\N$, where $K$ is a non-trivial finite abelian group. For $t\in\N_+$, $l\in\Z$, let $\Delta_{t,l}=\Delta K^{N_1+l}\oplus\ldots\oplus\Delta K^{N_t+l}$. Then, for a fixed $t\in\N_+$, and for every $k\in\N$, 
\begin{itemize}
\item[(a)]$\Delta_{t,0}+\Delta_{t,1}+\ldots+\Delta_{t,k}=\Delta_{t,0}\oplus\Delta_{t,1}\oplus\ldots\oplus\Delta_{t,k}$; and
\item[(b)]$\Delta_{t,0}+\Delta_{t,-1}+\ldots+\Delta_{t,-k}=\Delta_{t,0}\oplus\Delta_{t,-1}\oplus\ldots\oplus\Delta_{t,-k}$.
\end{itemize}
\end{lemma}
\begin{proof}
(a) We proceed by induction. For $k=1$, we have to prove that $\Delta_{t,0}\cap \Delta_{t,1}=\{0\}$.
Assume that $x\in \Delta_{t,1}= \Delta K^{N_1+1}\oplus\ldots\oplus\Delta K^{N_t+1}$.
By Remark \ref{supp(x)} $\supp(x)=Q_1\dot{\cup}\ldots\dot{\cup} Q_t$, where 
\begin{align*}
Q_1&=\begin{cases}\text{either} & (N_1+1)\setminus (N_2+1) \\ \text{or}& \emptyset \end{cases},\\
& \vdots \\ 
Q_{t-1}&=\begin{cases}\text{either} & (N_{t-1}+1)\setminus (N_t+1) \\ \text{or} & \emptyset \end{cases}, \\
Q_t&=\begin{cases}\text{either} & N_t+1\\ \text{or} & \emptyset\end{cases}.
\end{align*}
If also $x\in\Delta_{t,0}=\Delta K^{N_1}\oplus\ldots\oplus\Delta K^{N_t}$, then $\supp(x)\subseteq N_1$ and so, by Lemma \ref{subsetneq}(a), $Q_i=\emptyset$ for every $i\in\{1,\ldots,t\}$, that is $x=0$.

Suppose now that $k\geq 2$ and that $\Delta_{t,0}+\Delta_{t,1}+\ldots+\Delta_{t,k}=\Delta_{t,0}\oplus\Delta_{t,1}\oplus\ldots\oplus\Delta_{t,k}$.
We have to prove that $(\Delta_{t,0}\oplus\Delta_{t,1}\oplus\ldots\oplus\Delta_{t,k})\cap\Delta_{t,k+1}=\{0\}$.
Let $x\in\Delta_{t,k+1}=\Delta K^{N_1+(k+1)}\oplus\ldots\oplus\Delta K^{N_t+(k+1)}$. Then $\supp(x)=Q_1\dot{\cup}\ldots\dot{\cup} Q_t$, where
\begin{align*} 
Q_1&=\begin{cases}\text{either} & (N_1+(k+1))\setminus (N_2+(k+1)) \\ \text{or} & \emptyset \end{cases},\\
& \vdots \\
Q_{t-1}&=\begin{cases}\text{either} & (N_{t-1}+(k+1))\setminus (N_t+(k+1)) \\ \text{or} & \emptyset \end{cases}, \\
Q_t&=\begin{cases}\text{either} & N_t+(k+1) \\ \text{or} &\emptyset \end{cases}.
\end{align*}
If also $x\in\Delta_{t,0}\oplus\Delta_{t,1}\oplus\ldots\oplus\Delta_{t,k}=(\Delta K^{N_1}\oplus\ldots\oplus\Delta K^{N_t})\oplus(\Delta K^{N_1+1}\oplus\ldots\oplus\Delta K^{N_t+1})\oplus\ldots\oplus(\Delta K^{N_1+k}\oplus\ldots\oplus\Delta K^{N_t+k})$, then $\supp(x)\subseteq N_1\cup (N_1+1)\cup\ldots\cup N_1+k$ and so, by Lemma \ref{subsetneq}(a), $Q_i=\emptyset$ for every $i\in\{1,\ldots,t\}$, that is $x=0$. This concludes the proof.

\smallskip
(b) is analogous to (a).
\end{proof}

This proves that for every $t\in\N_+$ the families $\{\Delta_{t,k}:k\in\N\}$ and $\{\Delta_{t,-k}:k\in\N\}$ of finite subgroups of $K^\N$ are independent.

\section{Proof of Theorem \ref{mt}}\label{infinite}

In \cite{AADGH} the algebraic entropy of a generalized shift $\sigma_\lambda:K^\Gamma\to K^\Gamma$ restricted to the direct sum $\bigoplus_\Gamma K$ was computed precisely; we recall this result in Theorem \ref{mt-sum} below. As noted in the introduction, in this case we have to require that $\lambda$ has finite fibers, because this is equivalent to $\bigoplus_\Gamma K$ being a $\sigma_\lambda$-invariant subgroup of $K^\Gamma$.

In \eqref{aadgh-eq} below the algebraic entropy of $\sigma_\lambda\restriction_{\bigoplus_\Gamma K}$ is expressed as the product of the string number $s(\lambda)$ of $\lambda$ with the logarithm of the cardinality of the finite abelian group $K$. But while the algebraic entropy $\ent(-)$ is either a real number or the symbol $\infty$, the string number $s(-)$ is either a finite natural number or an infinite cardinal. Then for a self-map $\lambda:\Gamma\to \Gamma$ we introduce $s(\lambda)^*$ defined by $s(\lambda)^*=s(\lambda)$ if $s(\lambda)$ is finite and $s(\lambda)^*=\infty$ in case $s(\lambda)$ is infinite. 

\begin{theorem}\emph{\cite[Theorem 4.14]{AADGH}}\label{mt-sum}
Let $\Gamma$ be a set, $\lambda:\Gamma\to\Gamma$ a function such that $\lambda^{-1}(i)$ is finite for every $i\in\Gamma$, and $K$ a non-trivial finite abelian group. Then 
\begin{equation}\label{aadgh-eq}
\ent(\sigma_\lambda\restriction_{\bigoplus_\Gamma K})=s(\lambda)^*\cdot\log|K|.
\end{equation}
\end{theorem}

This theorem gives the idea of using strings also in the case of the calculation of the algebraic entropy of $\sigma_\lambda:K^\Gamma\to K^\Gamma$. Moreover, one of the main tools in proving this theorem was Remark 4.8 in \cite{AADGH}; the following proposition is its counterpart for $\sigma_\lambda:K^\Gamma\to K^\Gamma$.

\begin{proposition}\label{magic}
Let $\Gamma$ be a set, $\lambda:\Gamma\to\Gamma$ a function, and $K$ a non-trivial finite abelian group. Suppose that $\Gamma=\Gamma'\cup\Gamma''$ a partition of $\Gamma$ and that $\lambda^{-1}(\Gamma')\subseteq \Gamma'$ (i.e., $\lambda(\Gamma'')\subseteq\Gamma''$). Then $$\ent(\sigma_\lambda)=\ent(\sigma_{\lambda}\restriction_{K^{\Gamma'}})+\ent({\sigma_{\lambda\restriction_{\Gamma''}}}).$$ In particular, if $\Lambda$ is a $\lambda$-invariant subset of $\Gamma$, then $\ent(\sigma_\lambda)\geq\ent(\sigma_{\lambda\restriction_\Lambda})$.
\end{proposition}
\begin{proof}
By Lemma \ref{invariance}(a) $K^{\Gamma'}$ is $\sigma_\lambda$-invariant. Moreover, it is possible to consider $\lambda\restriction_{\Gamma''}:\Gamma''\to\Gamma''$.
Let $p_2:K^\Gamma=K^{\Gamma'}\oplus K^{\Gamma''}\to K^{\Gamma''}$ and $\pi:K^\Gamma\to K^\Gamma/K^{\Gamma'}$ be the canonical projections. Denote by $\xi:K^\Gamma/K^{\Gamma'}\to K^{\Gamma''}$ the (unique) isomorphism such that $p_2=\xi\circ\pi$. Finally, let $\overline{\sigma_\lambda}:K^\Gamma/K^{\Gamma'}\to K^\Gamma/K^{\Gamma'}$ be the homomorphism induced by $\sigma_\lambda$. Then $\overline{\sigma_\lambda}=\xi^{-1}\sigma_{\lambda\restriction_{\Gamma''}}\xi$.
The following diagram explains the situation.
\begin{equation*}
\xymatrix@-0.5pc{
K^{\Gamma'} \ar@{->}[rr]^{\sigma_\lambda\restriction_{K^{\Gamma'}}} \ar@{^{(}->}[d]& &  K^{\Gamma'}\ar@{^{(}->}[d]\\
K^\Gamma \ar@{->}[rr]^{\sigma_\lambda} \ar@{->>}[d]_{\pi} \ar@{-->}[ddr]^{p_2}& & K^\Gamma\ar@{->>}[d]_\pi \ar@{-->}[ddr]^{p_2}\\
K^\Gamma/K^{\Gamma'}\ar@{->}[rr]^{\overline{\sigma_\lambda}} \ar[dr]_{\xi} & & K^\Gamma/K^{\Gamma'} \ar[dr]_\xi\\
& K^{\Gamma''}\ar@{->}[rr]^{\sigma_{\lambda\restriction_{\Gamma''}}} & & K^{\Gamma''}
}
\end{equation*}
By Proposition \ref{properties}(b) $\ent(\overline{\sigma_\lambda})=\ent(\sigma_{\lambda\restriction_{\Gamma''}})$. Applying this equality and Theorem \ref{AT} we have the wanted equality.
\end{proof}

Applying this result in the following two lemmas, we see in particular that in case a function $\lambda$ is not bounded, then the algebraic entropy of the generalized shift $\sigma_\lambda$ is necessarily infinite.

\medskip
The next proposition shows that the algebraic entropy of the generalized shift $\sigma_\lambda$ is infinite in case $\lambda$ admits either a string or an infinite orbit. The proofs of (a) and (b) are similar and in both we apply the technical lemmas of Section \ref{disjoint}.

\begin{lemma}\label{asl>0->ent=infty}
Let $\Gamma$ be a set and $\lambda:\Gamma\to \Gamma$ a function. Let $K$ be a non-trivial finite abelian group and consider the generalized shift $\sigma_\lambda:K^\Gamma\to K^\Gamma$.
\begin{itemize}
\item[(a)]If $s(\lambda)>0$, then $\ent(\sigma_\lambda)=\infty$.
\item[(b)]If $o(\lambda)>0$, then $\ent(\sigma_\lambda)=\infty$.
\end{itemize}
\end{lemma}
\begin{proof}
(a) Let $S=\{s_n:n\in\N\}$ be a string of $\lambda$ in $\Gamma$; we can suppose without loss of generality that it is acyclic. Let $\Lambda=S\cup\{\lambda^n(s_0):n\in\N_+\}$. Then $\lambda(\Lambda)\subseteq \Lambda$. So let $\psi=\lambda\restriction_\Lambda$. By Proposition \ref{magic} $\ent(\sigma_\lambda)\geq\ent(\sigma_\psi)$, where $\sigma_\psi:K^\Lambda\to K^\Lambda$, and so it suffices to prove that $\ent(\sigma_\psi)=\infty$. 
For $m\in\N_+$ and $k\in\N$ let 
\begin{center}
$S_{m,k}=\{s_{n}:n\in N_m+k\},$ where $N_m+k$ is defined in \eqref{N_m}.
\end{center}
Fix $t\in\N_+$, and let $F_t=\Delta K^{S_{1,0}}+\ldots+\Delta K^{S_{t,0}}.$
For every $k\in\N$, by the definition of string and of $\psi$, and by Proposition \ref{G_F}(c),
$\sigma_\psi^k(\Delta K^{S_{m,0}})=\Delta K^{S_{m,k}}$ for every $m\in\N_+$,
and so 
$$\sigma_\psi^{k}(F_t)=\sigma_\psi^k(\Delta K^{S_{1,0}})+\ldots+\sigma_\psi^k(\Delta K^{S_{t,0}})=
\Delta K^{S_{1,k}}+\ldots+\Delta K^{S_{t,k}}.$$
By Lemma \ref{pseudo-eq1-lemma}(a) this sum is direct, that is,
$\sigma_\psi^k(F_t)=\Delta K^{S_{1,k}}\oplus\ldots\oplus\Delta K^{S_{t,k}}\cong K^t$ for every $k\in\N$.
By Lemma \ref{pseudo-eq2-lemma}(a) for every $k\in\N_+$ the sum $T_k(\sigma_\psi,F_t)=F_t+ \sigma_\psi(F_t)+\ldots+\sigma_\psi^{k-1}(F_t)$ is direct, that is,
$$T_{k}(\sigma_\psi,F_t)=F_t\oplus \sigma_\psi(F_t)\oplus\ldots\oplus\sigma_\psi^{k-1}(F_t)\cong K^{kt}.$$
Then $|T_k(\sigma_\psi,F_t)|=|K|^{k t}$ for every $k\in\N_+$ and so $H(\sigma_\psi,F_t)=t\log|K|$.
Since this can be done for every $t\in\N_+$, it follows that $\ent(\sigma_\psi)=\infty$.

\smallskip
(b) Let $A=\{a_n:n\in\N\}$ be an infinite orbit of $\lambda$. Then $\lambda(A)\subseteq A$; so let $\alpha=\lambda\restriction_A$. By Proposition \ref{magic} $\ent(\sigma_\lambda)\geq\ent(\sigma_\alpha)$, where $\sigma_\alpha:K^A\to K^A$, and so it suffices to prove that $\ent(\sigma_\alpha)=\infty$. 
For $m\in\N_+$ and $k\in\N$ let 
\begin{center}
$A_{m,k}=\{a_{n}:n\in N_m-k\},$ where $N_m-k$ is defined in \eqref{N_m}.
\end{center}
Fix $t\in\N_+$, and let $F_t=\Delta K^{A_{1,0}}+\ldots+\Delta K^{A_{t,0}}.$
For every $k\in\N$, by the definition of infinite orbit and of $\alpha$, and by Proposition \ref{G_F}(c),
$\sigma_\alpha^k(\Delta K^{A_{m,0}})=\Delta K^{A_{m,k}}$ for every $m\in\N_+$,
and so 
$$\sigma_\alpha^{k}(F_t)=\sigma_\alpha^k(\Delta K^{A_{1,0}})+\ldots+\sigma_\alpha^k(\Delta K^{A_{t,0}})=\Delta K^{A_{1,k}}+\ldots+\Delta K^{A_{t,k}}.$$
By Lemma \ref{pseudo-eq1-lemma}(b) this sum is direct, that is,
$\sigma_\alpha^k(F_t)=\Delta K^{A_{1,k}}\oplus\ldots\oplus\Delta K^{A_{t,k}}\cong K^t$ for every $k\in\N$.
By Lemma \ref{pseudo-eq2-lemma}(a) for every $k\in\N_+$ the sum $T_k(\sigma_\alpha,F_t)=F_t+\sigma_\alpha(F_t)+\ldots+\sigma_\alpha^{k-1}(F_t)$ is direct, that is,
$$T_{k}(\sigma_\alpha,F_t)=F_t\oplus \sigma_\alpha(F_t)\oplus\ldots\oplus\sigma_\alpha^{k-1}(F_t)\cong K^{kt}.$$
Then $|T_k(\sigma_\alpha,F_t)|=|K|^{kt}$ for every $k\in\N_+$ and so $H(\sigma_\alpha,F_t)=t\log|K|$.
Since this can be done for every $t\in\N_+$, it follows that $\ent(\sigma_\alpha)=\infty$.
\end{proof}

It is worthwhile noting that in \cite{DGSZ} the algebraic entropy of the Bernoulli shifts restricted to the direct sums was calculated, and in \cite{AADGH} it was described how the left Bernoulli shift $_K\beta$ and the two-sided Bernoulli shift $\overline\beta_K$ are generalized shifts, and how the right Bernoulli shift $\beta_K$ can be ``approximated" by a generalized shift with the same algebraic entropy:

\begin{example}\label{bernoulli}
Let $K$ be a non-trivial finite abelian group, and consider the Bernoulli shifts $\beta_K,\ _K\beta:K^\N\to K^\N$ and $\overline\beta_K:K^\Z\to K^\Z$ (defined in the introduction).
\begin{itemize}
\item[(a)] Then:
\begin{itemize}
\item[(a$_1$)] $_K\beta=\sigma_{\lambda_1}$, with $\lambda_1:\N\to\N$ defined by $n\mapsto n+1$ for every $n\in\N$;
\item[(a$_2$)] $\overline\beta_K=\sigma_{\lambda_2}$, with $\lambda_2:\Z\to\Z$ defined by $n\mapsto n-1$ for every $n\in\Z$;
\item[(a$_3$)] $\ent(\beta_K)=\ent(\sigma_{\lambda_3})$, where $\lambda_3:\N\to\N$ is defined by $n\mapsto n-1$ for every $n\in\N_+$ and $0\mapsto 0$, since $\beta_K\restriction_{K^{\N_+}}=\sigma_{\lambda_3}\restriction_{K^{\N_+}}$ and $K^\N/K^{\N_+}\cong K$ is finite --- so it is possible to apply Theorem \ref{AT}.
\end{itemize}
Note that $s(\lambda_1)=0$ and $o(\lambda_1)=1$, $s(\lambda_2)=o(\lambda_2)=1$, $s(\lambda_3)=1$ and $o(\lambda_3)=0$.
\item[(b)] It can be seen as a consequence of item (a) and Theorem \ref{mt-sum} that $$\ent(\beta_K\restriction_{\bigoplus_\N K})=\ent(\overline\beta_K\restriction_{\bigoplus_\Z K})=\log|K|$$ and $$\ent(_K\beta\restriction_{\bigoplus_\N K})=0.$$
\end{itemize}
\end{example}

Lemma \ref{asl>0->ent=infty}, together with this example, gives as a corollary the value of the algebraic entropy of the Bernoulli shifts considered on the direct products:

\begin{corollary}\label{beta}
Let $K$ be a non-trivial finite abelian group, and consider the Bernoulli shifts $\beta_K,\ _K\beta:K^\N\to K^\N$ and $\overline\beta_K:K^\Z\to K^\Z$. Then $$\ent(\beta_K)=\ent(_K\beta)=\ent(\overline\beta_K)=\infty.$$
\end{corollary}

Now we show that the algebraic entropy of a generalized shift $\sigma_\lambda$ is infinite also in case $\lambda$ has a ladder and in case $\lambda$ has a periodic ladder, that is, periodic orbits of arbitrarily large length. The technique used in the proof of this result is different from that used in the proof of Lemma \ref{asl>0->ent=infty}, and this is why we give them separately.

\begin{lemma}\label{ladder}\label{big:orbits}
Let $\Gamma$ be a set, $\lambda:\Gamma\to \Gamma$ a function, $K$ a non-trivial finite abelian group and consider the generalized shift $\sigma_\lambda:K^\Gamma\to K^\Gamma$. 
\begin{itemize}
\item[(a)] If $l(\lambda)>0$, then $\ent(\sigma_\lambda)=\infty$.
\item[(b)] If $p(\lambda)>0$, then $\ent(\sigma_\lambda)=\infty$.
\end{itemize}
\end{lemma}
\begin{proof}
(a) Let $L=\bigcup_{m\in\N}L_m\subseteq \Gamma\setminus\Per(\lambda)$ be a ladder of $\lambda$ in $\Gamma$, where each $L_m=\{l_{m,0},\ldots,l_{m, b_m}\}$. Let $\Lambda=L\cup\{\lambda^n(l_{m,0}):m\in\N, n\in\N_+\}$, which is $\lambda$-invariant and so define $\rho=\lambda\restriction_\Lambda$.  By Proposition \ref{magic} $\ent(\sigma_\lambda)\geq\ent(\sigma_\rho)$, where $\sigma_\rho:K^\Lambda\to K^\Lambda$, and so it suffices to prove that $\ent(\sigma_\rho)=\infty$. 
Let $\N=\bigcup_{i\in\N}N_i$ be a partition of $\N$ in infinitely many infinite subsets $N_i$ of $\N$. For each $i\in\N$ let $\Lambda_i=\bigcup_{m\in N_i}L_m$. Then each $\Lambda_i$ is a ladder of $\rho$ and $L=\bigcup_{i\in\N}\Lambda_i$ is a partition of $L$; so $K^L\cong\prod_{i\in\N}K^{\Lambda_i}$. Since each $\Lambda_i$ is $\rho^{-1}$-invariant, by Lemma \ref{invariance}(a) each $K^{\Lambda_i}$ is a $\sigma_\rho$-invariant subgroup of $K^L$. By Lemma \ref{uff}(a) $\ent(\sigma_\rho\restriction_{K^{\Lambda_i}})>0$ for every $i\in\N$ and so Lemma \ref{invariance}(b) implies that $\ent(\sigma_\lambda)=\infty$.

\smallskip
(b) Let $P=\bigcup_{n\in\N_+}P_n$ be a periodic ladder of $\lambda$ in $\Gamma$.
Since $\lambda(P)\subseteq P$, let $\phi=\lambda\restriction_P$; by Proposition \ref{magic} $\ent(\sigma_\lambda)\geq\ent(\sigma_\phi)$ and so it suffices to prove that $\sigma_\phi:K^P\to K^P$ has $\ent(\sigma_\phi)=\infty$.
Let $\N_+=\bigcup_{i=1}^\infty N_i$ be a partition of $\N_+$ such that each $N_i$ is infinite, and let $\Lambda_i=\bigcup_{n\in N_i}P_n$. Consequently $P=\bigcup_{i=1}^\infty\Lambda_i$ is a partition of $P$, and so $K^P\cong\prod_{i=1}^\infty K^{\Lambda_i}$. For every $i\in\N_+$, $\phi^{-1}(\Lambda_i)\subseteq\Lambda_i$, so by Lemma \ref{invariance}(a) each $K^{\Lambda_i}$ is a $\sigma_\phi$-invariant subgroup of $K^P$. 
 By Lemma \ref{uff}(b) $\ent(\sigma_\rho\restriction_{K^{\Lambda_i}})>0$ for every $i\in\N$ and so Lemma \ref{invariance}(b) implies that $\ent(\sigma_\lambda)=\infty$.
\end{proof}

Thanks to the characterization of bounded functions given by Theorem \ref{bounded}, and in view of the preceding results, we can now prove Theorem \ref{mt}.

\begin{proof}[\bf Proof of Theorem \ref{mt}]
(c)$\Leftrightarrow$(d)$\Leftrightarrow$(e) is Theorem \ref{pre-mt}, while (e)$\Leftrightarrow$(a) is given by Proposition \ref{properties}(a),
and 
(a)$\Rightarrow$(b) is obvious.

\smallskip
(b)$\Rightarrow$(c) Assume that $\lambda$ is not bounded. By Theorem \ref{bounded} this happens if at least one of $s(\lambda)$, $o(\lambda)$, $l(\lambda)$, $p(\lambda)$ is non-zero.
Then apply respectively (a) or (b) of Lemma \ref{asl>0->ent=infty}, or (a) or (b) of Lemma \ref{ladder}.
\end{proof}

Note that among the equivalent conditions of Theorem \ref{mt} it is not possible to add that $\sigma_\lambda\restriction_{\bigoplus_\Gamma K}$ has entropy zero, even when $\sigma_\lambda$ is an automorphism:

\begin{example}
Let $\Gamma$ be a countably infinite set and $\lambda:\Gamma\to \Gamma$ a function such that $\Gamma$ is a periodic ladder of $\lambda$.
Then $\Gamma=\Per(\lambda)$ and $\lambda$ is a bijection, which is locally periodic, non-periodic (and so non-quasi-periodic). While $\ent(\sigma_\lambda)=\infty$ by Lemma \ref{big:orbits}(b), $\ent(\sigma_\lambda\restriction_{\bigoplus_\Gamma K})=0$.
\end{example}

Theorem \ref{mt} can be generalized replacing finite abelian groups by arbitrary torsion abelian groups:

\begin{corollary}
Let $\Gamma$ be a set, $\lambda:\Gamma\to\Gamma$ a function, $K$ a non-trivial torsion abelian group and consider the generalized shift $\sigma_\lambda:K^\Gamma\to K^\Gamma$. Then $\ent(\sigma_\lambda)=0$ if and only if $\lambda$ is bounded, otherwise $\ent(\sigma_\lambda)=\infty$.
\end{corollary}
\begin{proof}
By Proposition \ref{properties}(c) 
\begin{align*}
\ent(\sigma_\lambda)&=\sup\{\ent(\sigma_\lambda\restriction_{F^{(\Gamma)}}):F\ \text{is a finite subgroup of}\ K\} \\
&=\sup\{\ent(\sigma_{\lambda,F}):F\ \text{is a finite subgroup of}\ K\}.
\end{align*}
By Theorem \ref{mt} $\ent(\sigma_{\lambda,F})=0$ if and only if $\lambda$ is bounded and otherwise $\ent(\sigma_{\lambda,F})=\infty$.
\end{proof}

\begin{remark}
Let $\Gamma$ be a set, $\lambda:\Gamma\to\Gamma$ a function and $K$ a non-trivial finite abelian group.
In \cite{AADGH} the set $\Gamma^+=\bigcap_{n\in\N_+}\lambda^n(\Gamma)$ was defined. 
\begin{itemize}
\item[(a)] The set $\Gamma^+$ was useful in computing the algebraic entropy of the restriction of a generalized shift to the direct sum, that is, of $\sigma_\lambda\restriction_{\bigoplus_\Gamma K}:\bigoplus_\Gamma K\to\bigoplus_\Gamma K$. Indeed, in order to consider this restriction, $\lambda$ has to have $\lambda^{-1}(i)$ finite for every $i\in \Gamma$ and in this case $\lambda\restriction_{\Gamma^+}:\Gamma^+\to \Gamma^+$ is surjective and $\ent(\sigma_\lambda\restriction_{\bigoplus_\Gamma K})=\ent(\sigma_{\lambda\restriction_{\Gamma^+}}\restriction_{\bigoplus_{\Gamma^+}K})$.
\item[(b)] In general for a function $\lambda:\Gamma\to \Gamma$ it is not true that its restriction to $\Gamma^+$ is surjective. Take for example $\Gamma=\{g,h\}\cup\bigcup_{n\in\N}\Gamma_n$, where for every $n\in\N$, $\Gamma_n=\{g_{n,0},\ldots,g_{n,n}\}$, $\lambda(g_{n,l})=g_{n,l-1}$ for every $l\in\{1,\ldots,n\}$, $\lambda(g_{n,0})=g$, $\lambda(g)=h$ and $\lambda(h)=h$. Then $\Gamma^+=\{g,h\}$, but $g\not\in\lambda(\Gamma^+)=\{h\}$.
\item[(c)] In general it is not true that $\ent(\sigma_\lambda)=\ent(\sigma_{\lambda\restriction_{\Gamma^+}})$, because for example if $\lambda$ admits a ladder $L$  in $\Gamma$, and $\Gamma=L\cup\Per_n(\lambda)$ for some $n\in\N_+$, then $\Gamma^+=\Per_n(\lambda)$, and so $\ent(\sigma_{\lambda\restriction_{\Gamma^+}})=0$, while $\ent(\sigma_\lambda)=\infty$ by Theorem \ref{mt}.
\item[(d)] Observe that the function considered in (c) is not surjective. In fact, it is clear that $\lambda$ is surjective if and only if $\Gamma=\Gamma^+$.\end{itemize}
\end{remark}

\medskip
We explain now in detail how Theorem \ref{mt} solves Problems 6.1 and 6.2 in \cite{AADGH}. Indeed, as asked in the first part of Problem 6.1, it gives the precise value of the algebraic entropy of a generalized shift $\sigma_\lambda$ (in particular this answers negatively Problem 6.2(b), which asked if it was possible that $0<\ent(\sigma_\lambda)<\infty$).

Moreover, Example \ref{bernoulli} and Corollary \ref{beta} answer negatively the question in Problem 6.1, showing that in general it is not true that $\ent(\sigma_\lambda)$ coincides with $\ent(\sigma_\lambda\restriction_{\bigoplus_\Gamma K})$. Indeed, the left Bernoulli shift $_K\beta$ is a generalized shift and has $\ent(_K\beta)=\infty$ by Corollary \ref{beta}, while $\ent(_K\beta\restriction_{\bigoplus_\N K})=0$ by Example \ref{bernoulli}.
This shows also that it is possible that $\ent(\sigma_\lambda\restriction_{\bigoplus_\Gamma K})=0$, while $\ent(\sigma_\lambda)>0$, which was asked in Problem 6.2(a).

\section{Open problems}

For a set $\Gamma$, a function $\lambda:\Gamma\to\Gamma$ and a non-trivial abelian group $K$, we can consider $K^\Gamma$ endowed with the product topology of the discrete topologies on $K$. In this way $K^\Gamma$ is a compact abelian group, and the generalized shift $\sigma_\lambda:K^\Gamma\to K^\Gamma$ is continuous. In relation to Theorem \ref{mt}, the following question arises.

\begin{problem}
Let $K$ be a non-trivial finite abelian group. Does there exist a continuous endomorphism $\phi:K^\N\to K^\N$ with $0<\ent(\phi)<\infty$?
\end{problem}

We conclude the paper by setting the following problem for infinite orbits and $o(-)$, which is similar to Problem 6.6 in \cite{AADGH} for strings and $s(-)$.

\begin{problem}
Let $\Gamma$ be an abelian group and $\lambda:\Gamma\to\Gamma$ a group endomorphism. Calculate $o(\lambda)$. In particular, is it true that $o(\lambda)>0$ implies $o(\lambda)$ infinite?
\end{problem}

\end{document}